\documentclass[10pt, a4paper]{amsart}
\usepackage{amsthm}
\usepackage[dvipdfmx]{graphicx}
\usepackage{caption}
\usepackage{float}
{\theoremstyle{definition}
\newtheorem{dfn}{Definition}}
\newtheorem{prop}[dfn]{Proposition}

{\theoremstyle{definition}
}
\newtheorem{lem}[dfn]{Lemma}

\newtheorem{conj}[dfn]{Conjecture}
{\theoremstyle{definition}
\newtheorem{exa}[dfn]{Example}}

\usepackage[top=20truemm,bottom=20truemm,left=25truemm,right=25truemm]{geometry}
\usepackage{amsmath,amssymb,amscd,bm,graphicx}
\usepackage{color}
\usepackage{comment}
\usepackage[colorlinks=true,
    citecolor=denim,
    linkcolor=alizarin,
    urlcolor=lightseagreen]{hyperref}
\definecolor{alizarin}{rgb}{0.82, 0.1, 0.26}
\definecolor{azure(colorwheel)}{rgb}{0.0, 0.5, 1.0}
\definecolor{blue(pigment)}{rgb}{0.2, 0.2, 0.6}
\definecolor{denim}{rgb}{0.08, 0.38, 0.74}
\definecolor{mint}{rgb}{0.24, 0.71, 0.54}
\definecolor{parisgreen}{rgb}{0.31, 0.78, 0.47}
\definecolor{persiangreen}{rgb}{0.0, 0.65, 0.58}
\definecolor{seagreen}{rgb}{0.18, 0.55, 0.34}
\definecolor{shamrockgreen}{rgb}{0.0, 0.62, 0.38}
\definecolor{green(pigment)}{rgb}{0.0, 0.65, 0.31}
\definecolor{cadmiumgreen}{rgb}{0.0, 0.42, 0.24}
\definecolor{lightseagreen}{rgb}{0.13, 0.7, 0.67}
\definecolor{mediumseagreen}{rgb}{0.24, 0.7, 0.44}
\definecolor{pinegreen}{rgb}{0.0, 0.47, 0.44}
\definecolor{tealgreen}{rgb}{0.0, 0.51, 0.5}

\newcommand{\bC}{\mathbb{C}}

\newcommand{\bZ}{\mathbb{Z}}

\begin{document}

\title{ASYMPTOTIC BEHAVIOR OF TWISTED ALEXANDER INVARIANTS
FOR HYPERBOLIC KNOTS WITH AT MOST SIX CROSSINGS}

% The first author
\author[A. Aso]{Airi Aso}
\address{Mathematical Institute, Tohoku University,
6--3, Aramaki Aza--Aoba, Aoba--ku,Sendai 980--8578, JAPAN} 
\email{airi.aso.c3@tohoku.ac.jp}

% Abstract
\begin{abstract}
Let $K$ be a hyperbolic knot and let $\rho_n$ be the $n$-dimensional irreducible representation induced from a lift of its holonomy representation. Motivated by Goda's asymptotic volume formula and the complexified Volume Conjecture, we study whether the higher-dimensional twisted Alexander invariants associated with $\rho_n$ detect the complex volume of the knot complement.

We compute
$$
\frac{\pi}{2}
\log
\left(
\frac{A_{K,n-2}(1)A_{K,n+2}(1)}
{A_{K,n}(1)^2}
\right)
$$
for all hyperbolic knots with at most six crossings. Our numerical experiments indicate that these values approach
$$
\operatorname{Vol}(S^3\setminus K)
+i\,2\pi^2\operatorname{CS}(S^3\setminus K)
$$
modulo $i\pi^2\bZ$. Based on these computations, we propose a complexified analogue of Goda's asymptotic volume formula.
\end{abstract}

% \maketitle is after abstract
\maketitle

\section{Introduction}

The twisted Alexander polynomial is a generalization of the classical Alexander polynomial obtained by incorporating a linear representation of the knot group. Twisted Alexander polynomials were introduced in related forms by Lin\cite{Lin} and Wada\cite{Wada}. Lin considered twisted Alexander polynomials associated with representations of knot groups, while Wada gave a definition for finitely presentable groups. In the present paper, we use Wada's definition of the twisted Alexander polynomial. 

An important relation between twisted Alexander polynomials and Reidemeister torsion was established by Kitano, who showed that, for knot groups, Wada's twisted Alexander polynomial can be interpreted in terms of Reidemeister torsion \cite{Kitano}. Since then, twisted Alexander polynomials have been extensively studied and have proved useful in detecting topological properties of knots and \(3\)-manifolds, such as fiberedness and the Thurston norm. For hyperbolic knots, Dunfield, Friedl and Jackson studied twisted Alexander polynomials associated with a lift of the holonomy representation and demonstrated through extensive computations that these invariants contain strong topological information \cite{DFJ}.

The relation with hyperbolic geometry becomes particularly significant when one considers higher-dimensional representations induced from the holonomy representation. Let \(K\) be a hyperbolic knot and let
\[
\rho:\pi_1(S^3\setminus K)\longrightarrow SL(2,\mathbb C)
\]
be a lift of its holonomy representation. For each \(n\geq 2\), let
$$
\sigma_n:SL(2,\mathbb C)\longrightarrow SL(n,\mathbb C)
$$
be the \(n\)-dimensional irreducible representation, and put
$$
\rho_n=\sigma_n\circ\rho.
$$
Menal-Ferrer and Porti studied the higher-dimensional Reidemeister torsion invariants associated with these representations and showed that their asymptotic behavior is governed by the hyperbolic volume of the \(3\)-manifold \cite{MenalFerrerPorti}.

On the other hand, Porti studied non-acyclic Reidemeister torsion associated with the adjoint representation of the holonomy representation \cite{Porti}, and Yamaguchi described its relation with the behavior of the corresponding twisted Alexander invariant at \(t=1\) \cite{Yamaguchi}. 
Goda observed that Yamaguchi's method can be applied to the higher-dimensional Reidemeister torsions of Menal-Ferrer and Porti and obtained a volume formula in terms of twisted Alexander invariants \cite{Goda}.

We recall Goda's formula. Let \(\Delta_{K,\rho_n}(t)\) denote the twisted Alexander invariant associated with \(\rho_n\) in Wada's notation. For \(k>1\), set
$$
A_{K,2k}(t)
=
\frac{\Delta_{K,\rho_{2k}}(t)}
     {\Delta_{K,\rho_2}(t)}
\qquad\text{and}\qquad
A_{K,2k+1}(t)
=
\frac{\Delta_{K,\rho_{2k+1}}(t)}
     {\Delta_{K,\rho_3}(t)}.
$$
Then Goda proved
\begin{equation}\label{eq Goda}
\lim_{k\to\infty}
\frac{\log|A_{K,2k+1}(1)|}{(2k+1)^2}
=
\lim_{k\to\infty}
\frac{\log|A_{K,2k}(1)|}{(2k)^2}
=
\frac{\operatorname{Vol}(K)}{4\pi},
\end{equation}
where \(\operatorname{Vol}(K)\) denotes the hyperbolic volume of \(S^3\setminus K\). In the odd-dimensional case, the numerator and denominator have a common zero at \(t=1\), and \(A_{K,2k+1}(1)\) is understood as the value of the resulting quotient after cancellation.
The appearance of the hyperbolic volume in (1) is reminiscent of the Volume Conjecture for the colored Jones polynomial \cite{MurakamiMurakami}. Let \(J_N(K;q)\) denote the \(N\)-colored Jones polynomial of a hyperbolic knot \(K\). The Volume Conjecture predicts
\begin{equation}\label{eq MM}
\lim_{N\to\infty}
\frac{2\pi
\log
\left|
J_N\left(K;e^{2\pi i/N}\right)
\right|}
{N}
=
\operatorname{Vol}(S^3\setminus K).
\end{equation}
Thus, in both the colored Jones polynomial and the twisted Alexander invariant, the hyperbolic volume appears as the leading term of an asymptotic formula involving the absolute value of a knot invariant.

Moreover, a complexification of the Volume Conjecture was proposed in which the absolute value is removed and the Chern--Simons invariant appears together with the hyperbolic volume \cite{MurakamiMurakamiOkamotoTakataYokota}. With the normalization used in the present paper, this takes the form
\begin{equation}\label{eq MMYetal}
\lim_{N\to\infty}
\frac{2\pi
\log
J_N\left(K;e^{2\pi i/N}\right)}
{N}
=
\operatorname{Vol}(S^3\setminus K)
+
i\,2\pi^2 CS(S^3\setminus K)
\pmod {i\pi^2\mathbb Z}.
\end{equation}
The quantity
\begin{equation*}
\mathbb V(K)
=
\operatorname{Vol}(S^3\setminus K)
+
i\,2\pi^2 CS(S^3\setminus K)
\end{equation*}
will be referred to as the complex volume of the knot complement, with the above ambiguity understood.

Equations (1)--(3) suggest a natural question. Goda's formula extracts the hyperbolic volume from the asymptotic behavior of the absolute values \(|A_{K,n}(1)|\). What happens if the absolute value is removed and the complex values \(A_{K,n}(1)\) themselves are considered? In view of the complexification of the Volume Conjecture, it is natural to ask whether the phase of the higher-dimensional twisted Alexander invariants contains the Chern--Simons invariant and whether a complexification of Goda's formula detects the complex volume.

A direct replacement of \(\log|A_{K,n}(1)|\) by \(\log A_{K,n}(1)\), however, does not immediately give the desired asymptotic expression because the complex logarithm is multivalued. In particular, if one always takes the principal branch, the imaginary part of
$$
\frac{4\pi}{n^2}\log A_{K,n}(1)
$$
is bounded in absolute value by \(4\pi^2/n^2\) and therefore tends to zero. Thus the principal logarithm cannot directly retain a nonzero Chern--Simons term in this normalization.
We instead consider a second difference. Suppose formally that
\begin{equation*}
A_{K,n}(1)
\sim
\exp\left(
\frac{n^2}{4\pi}\mathbb V(K)
\right).
\end{equation*}
Then
$$
\frac{A_{K,n-2}(1)A_{K,n+2}(1)}
     {A_{K,n}(1)^2}
\sim
\exp\left(
\frac{2}{\pi}\mathbb V(K)
\right),
$$
since
$$
(n-2)^2+(n+2)^2-2n^2=8.
$$
Consequently, the quantity
\begin{equation}\label{AA}
\frac{\pi}{2}
\log
\left(
\frac{A_{K,n-2}(1)A_{K,n+2}(1)}
     {A_{K,n}(1)^2}
\right)
\end{equation}
is expected to recover the complex volume. Moreover, changing the branch of the logarithm changes \ref{AA} by an element of
\(
i\pi^2\mathbb Z,
\)
which agrees precisely with the ambiguity occurring in the complex volume.

In this paper, we investigate this expectation by explicit computer experiments. We compute the higher-dimensional twisted Alexander invariants associated with the holonomy representations for all hyperbolic knots with at most six crossings,
$$
4_1,\qquad 5_2,\qquad 6_1,\qquad 6_2,\qquad 6_3.
$$
Our computations indicate that, for both the even- and odd-dimensional sequences, the values in \ref{AA} approach the complex volume of the corresponding knot complement. In particular, the real parts approach the hyperbolic volumes, while the imaginary parts agree numerically with the corresponding Chern--Simons invariants modulo the expected ambiguity.
These observations lead us to the following conjecture.

\begin{conj}%main result
Let \(K\) be a hyperbolic knot in \(S^3\). For each fixed parity,
\begin{equation}\label{conj A}
\lim_{n\to\infty}
\frac{\pi}{2}
\log
\left(
\frac{A_{K,n-2}(1)A_{K,n+2}(1)}
     {A_{K,n}(1)^2}
\right)
=
\mathbb V(K)
\qquad
\text{in }
\mathbb C/i\pi^2\mathbb Z.
\end{equation}
Equivalently,
$$
\lim_{n\to\infty}
\frac{\pi}{2}
\log
\left(
\frac{A_{K,n-2}(1)A_{K,n+2}(1)}
     {A_{K,n}(1)^2}
\right)
\equiv
\operatorname{Vol}(S^3\setminus K)
+
i\,2\pi^2CS(S^3\setminus K)
\pmod{i\pi^2\mathbb Z},
$$
where the limit is taken separately through even and odd integers.
\end{conj}

Thus, the numerical evidence presented here suggests a close parallel between the Volume Conjecture and the asymptotic behavior of higher-dimensional twisted Alexander invariants. In both settings, taking absolute values yields the hyperbolic volume, whereas retaining the complex phase appears to recover the full complex volume.

The paper is organized as follows. We first recall Wada's definition of the twisted Alexander polynomial. We then review the irreducible \(n\)-dimensional representations of \(SL(2,\mathbb C)\) and the higher-dimensional representations \(\rho_n\) induced from a lift of the holonomy representation, and give an explicit formula for their matrix entries. Using these formulas, we compute the twisted Alexander invariants for the hyperbolic knots with at most six crossings. Finally, we study the asymptotic behavior of the normalized invariants \(A_{K,n}(1)\), compare the resulting numerical values with the hyperbolic volumes and Chern--Simons invariants of the knot complements, and present the numerical evidence for Conjecture 1.

\section{Acknowledgments}

The author would like to thank Professor Yoshiyuki Yokota for suggesting the initial idea that led to this work. The author is also grateful to Professor Hirotaka Akiyoshi for pointing out the ambiguity arising from the choice of branch of the complex logarithm. The author would also like to thank Professor Hiroshi Goda for helpful discussions and for answering many questions concerning his work on higher-dimensional twisted Alexander invariants.

\section{Twisted Alexander polynomials}

In this section, we briefly recall Wada's definition of the twisted Alexander polynomial for a knot.
Let $K$ be a knot in $S^3$, and let
$$
G(K):=\pi_1(S^3\setminus K)
=
\left\langle
x_1,\ldots,x_m
\ \middle|\
r_1,\ldots,r_{m-1}
\right\rangle
$$
be a deficiency-one presentation of the knot group. Let $F_m$ be the free group generated by $x_1,\ldots,x_m$, and let
$$
\phi:F_m\longrightarrow G(K)
$$
be the natural epimorphism.
Let
$$
\mathfrak a:G(K)\longrightarrow \langle t\rangle
$$
be the abelianization homomorphism, where $\langle t\rangle$ denotes the infinite cyclic multiplicative group generated by $t$. We choose the generators so that
$$
\mathfrak a(x_j)=t
$$
for each $j$.

Let
$$
\rho:G(K)\longrightarrow SL_d(\bC)
$$
be a $d$-dimensional representation. Extending $\phi$, $\mathfrak a$, and $\rho$ linearly to the corresponding group rings, we obtain a ring homomorphism
$$
\Phi:
\bZ F_m
\longrightarrow
M_d\bigl(\bC[t^{\pm1}]\bigr)
$$
defined by
$$
\Phi
=
(\rho\otimes\mathfrak a)\circ\phi.
$$

For $1\leq i\leq m-1$ and $1\leq j\leq m$, put
$$
A_{i,j}
=
\Phi\left(
\frac{\partial r_i}{\partial x_j}
\right)
\in
M_d\bigl(\bC[t^{\pm1}]\bigr),
$$
where
$$
\frac{\partial}{\partial x_j}:
\bZ F_m\longrightarrow\bZ F_m
$$
denotes the Fox derivative with respect to $x_j$.

The resulting Alexander matrix is the $(m-1)\times m$ block matrix
$$
M_\rho
=
\begin{pmatrix}
A_{1,1} & \cdots & A_{1,m}\\
\vdots  &        & \vdots\\
A_{m-1,1} & \cdots & A_{m-1,m}
\end{pmatrix}.
$$
For an index $k$ such that
$$
\det\Phi(x_k-1)\neq0,
$$
let $M_{\rho,k}$ be the $d(m-1)\times d(m-1)$ matrix obtained from $M_\rho$ by deleting the $k$-th block column.

Then Wada's twisted Alexander polynomial associated with $\rho$ is defined by
$$
\Delta_{K,\rho}(t)
=
\frac{\det M_{\rho,k}}
{\det\Phi(x_k-1)}.
$$
up to multiplication by a monomial factor. In the numerical computations below, we use the representative determined by the fixed presentation and the chosen lift of the holonomy representation.

\section{Holonomy representations and higher-dimensional representations}

\subsection{Holonomy representations}

Let $K$ be a hyperbolic knot and put
$$
M_K=S^3\setminus K.
$$
The complete hyperbolic structure on $M_K$ determines a developing map
$$
\operatorname{dev}:\widetilde{M_K}\longrightarrow \mathbb H^3
$$
and a holonomy representation
$$
\operatorname{Hol}_K:
G(K)
\longrightarrow
\operatorname{Isom}^{+}(\mathbb H^3)
\cong
PSL_2(\bC).
$$
The holonomy representation is discrete and faithful, and it is determined up to conjugation in $PSL_2(\bC)$.

In our computations, explicit holonomy representations are obtained from ideal triangulations of the knot complements. 
We identify the tetrahedra in the universal cover with ideal tetrahedra in $\mathbb H^3$ by means of the developing map. 
By tracing the images of the initial and terminal points associated with the Wirtinger generators, we obtain the corresponding M\"obius transformations and hence explicit matrices representing the holonomy. The resulting matrices for each knot will be given in the numerical computations below.

The holonomy representation admits a lift to $SL_2(\bC)$. Such lifts are naturally associated with spin structures on $M_K$. 
Throughout this paper, we choose the lift
\[
\rho: G(K)\longrightarrow SL_2(\bC)
\]
for which the image of a meridian has trace $2$.
Since the Wirtinger generators are conjugate meridians, their images
also have trace $2$. Each of these images is a nontrivial parabolic
element and hence is conjugate in $SL_2(\bC)$ to
\[
\begin{pmatrix}
1&1\\
0&1
\end{pmatrix}.
\]
All the explicit holonomy matrices used in the computations below are taken with this convention.

\subsection{Higher-dimensional irreducible representations}

For $n\geq2$, let
$$
V_n
=
\left\langle
X^{n-1},X^{n-2}Y,\ldots,XY^{n-2},Y^{n-1}
\right\rangle
\subset\bC[X,Y]
$$
be the complex vector space of homogeneous polynomials in two variables $X$ and $Y$ of degree $n-1$. We denote its standard ordered basis by
$$
\mathcal B_n
=
\left(
X^{n-1},X^{n-2}Y,\ldots,XY^{n-2},Y^{n-1}
\right).
$$

For
$A\in SL_2(\bC)$ and $p\in V_n$,
define an action of $A$ on $V_n$ by
$$
(A\cdot p)
\begin{pmatrix}
X\\
Y
\end{pmatrix}
=
p\left(
A^{-1}
\begin{pmatrix}
X\\
Y
\end{pmatrix}
\right).
$$
This action gives the $n$-dimensional irreducible representation
$$
\sigma_n:
SL_2(\bC)
\longrightarrow
SL_n(\bC),
$$
where $\sigma_n(A)$ denotes the matrix of the linear transformation
$$
A\cdot(-):V_n\longrightarrow V_n
$$
with respect to the basis $\mathcal B_n$.

Composing the chosen lift of the holonomy representation with $\sigma_n$, we obtain
$$
\rho_n
=
\sigma_n\circ\rho:
G(K)
\longrightarrow
SL_n(\bC).
$$
These are the representations used to define the higher-dimensional twisted Alexander invariants considered in this paper.

\begin{exa}
Let $n=3$ and
$$
A=
\begin{pmatrix}
1&0\\
1&1
\end{pmatrix}
\in SL_2(\bC).
$$
Since
$$
A^{-1}
\begin{pmatrix}
X\\
Y
\end{pmatrix}
=
\begin{pmatrix}
X\\
Y-X
\end{pmatrix},
$$
we have
$$
X(Y-X)=-X^2+XY
$$
and
$$
(Y-X)^2=X^2-2XY+Y^2.
$$
Therefore, with respect to the basis
$$
\mathcal B_3=(X^2,XY,Y^2),
$$
we obtain
$$
\sigma_3(A)
=
\begin{pmatrix}
1&-1&1\\
0&1&-2\\
0&0&1
\end{pmatrix}.
$$
\end{exa}

For the explicit computations of the twisted Alexander invariants, it is useful to have a formula for the entries of $\sigma_n(A)$.

\begin{lem}\label{rho_n}
Let
$$
A=
\begin{pmatrix}
a&b\\
c&d
\end{pmatrix}
\in SL_2(\bC).
$$
Then
$$
\sigma_n(A)=(a_{ij})\in SL_n(\bC),
$$
where
$$
a_{ij}
=
\sum_{k\in I_{i,j}}
\frac{
(n-j)!(j-1)!(-1)^{i+j}
a^{i-k-1}b^kc^{j-i+k}d^{n-j-k}
}{
k!(n-j-k)!(i-k-1)!(j-i+k)!
},
$$
and
$$
I_{i,j}
=
\left\{
k\in\bZ
\ \middle|\
\max\{0,i-j\}
\leq k
\leq
\min\{i-1,n-j\}
\right\}.
$$
\end{lem}

\begin{proof}
The $(i,j)$-entry $a_{ij}$ is the coefficient of
$$
X^{n-i}Y^{i-1}
$$
in the image of the $j$-th basis element
$$
X^{n-j}Y^{j-1}
$$
under $\sigma_n(A)$.

Since
$$
A^{-1}
=
\begin{pmatrix}
d&-b\\
-c&a
\end{pmatrix},
$$
we have
$$
\sigma_n(A)
\left(X^{n-j}Y^{j-1}\right)
=
(dX-bY)^{n-j}(-cX+aY)^{j-1}.
$$

Expanding the two factors gives
$$
(dX-bY)^{n-j}
=
\sum_{k=0}^{n-j}
\frac{(n-j)!}{k!(n-j-k)!}
(-1)^k
b^k d^{n-j-k}
X^{n-j-k}Y^k
$$
and
$$
(-cX+aY)^{j-1}
=
\sum_{\ell=0}^{j-1}
\frac{(j-1)!}{\ell!(j-1-\ell)!}
(-1)^{j-1-\ell}
a^\ell c^{j-1-\ell}
X^{j-1-\ell}Y^\ell.
$$
A term in the product contributes to the coefficient of
$X^{n-i}Y^{i-1}$ precisely when
$$
k+\ell=i-1.
$$
Thus
$$
\ell=i-1-k.
$$
The conditions
$$
0\leq k\leq n-j
\qquad\text{and}\qquad
0\leq\ell\leq j-1
$$
are equivalent to
$$
\max\{0,i-j\}
\leq k
\leq
\min\{i-1,n-j\}.
$$
Substituting $\ell=i-1-k$ into the coefficient of the product yields
$$
a_{ij}
=
\sum_{k\in I_{i,j}}
\frac{
(n-j)!(j-1)!(-1)^{i+j}
a^{i-k-1}b^kc^{j-i+k}d^{n-j-k}
}{
k!(n-j-k)!(i-k-1)!(j-i+k)!
},
$$
as required.
\end{proof}

\section{Numerical experiments}

In this section, we present numerical experiments for hyperbolic knots with at most six crossings. We begin with the figure-eight knot $4_1$ as a preliminary example. Since the Chern--Simons invariant of its complement vanishes, this example does not exhibit the nontrivial imaginary part which is the main feature of the complexification considered in this paper.

We then study the knot $5_2$ in detail. This is the first example among the knots considered here for which a nonzero Chern--Simons term appears, and we use this example to explain the computational procedure. Finally, we apply the same method to the hyperbolic knots with six crossings, namely $6_1$, $6_2$, and $6_3$.

For convenience, put
$$
Q_{K,n}
:=
\frac{\pi}{2}
\log
\left(
\frac{
A_{K,n-2}(1)A_{K,n+2}(1)
}{
A_{K,n}(1)^2
}
\right).
$$
We briefly recall why this second-difference expression is used. If one directly computes
$$
\frac{4\pi}{n^2}\log A_{K,n}(1)
$$
using the principal branch of the logarithm, 
then
\[
\left|
\operatorname{Im}
\left(
\frac{4\pi}{n^2}\log A_{K,n}(1)
\right)
\right|
\leq
\frac{4\pi^2}{n^2}.
\]
Hence its imaginary part necessarily tends to zero as $n\to\infty$.
 Thus the principal logarithm cannot retain a nonzero asymptotic Chern--Simons term.

On the other hand, if one formally expects
$$
A_{K,n}(1)
\sim
\exp
\left(
\frac{n^2}{4\pi}
\{
\operatorname{Vol}(S^3\setminus K)+ic
\}
\right),
$$
then
$$
\frac{
A_{K,n-2}(1)A_{K,n+2}(1)
}{
A_{K,n}(1)^2
}
\sim
\exp
\left(
\frac{2}{\pi}
\{
\operatorname{Vol}(S^3\setminus K)+ic
\}
\right).
$$
This motivates the definition of $Q_{K,n}$.

The even- and odd-dimensional sequences are considered separately, so that $n-2$, $n$, and $n+2$ always have the same parity. As explained in the Introduction, the ambiguity of the logarithm changes $Q_{K,n}$ by an element of $i\pi^2\mathbb Z$. We therefore regard $Q_{K,n}$ as an element of
$
\bC/i\pi^2\bZ.
$

For the comparison with the complex volume, we use Yokota's
potential-function method associated with ideal triangulations of knot
complements \cite{Yokota}.
The edge-gluing and completeness equations generally admit several
algebraic solutions. We choose the solution corresponding to the
complete hyperbolic structure. In our computations, this geometric
solution is identified by the maximal positive oriented volume computed
using the Bloch--Wigner dilogarithm \cite{Cho}.
The complex volume is then obtained from the potential-function formula
for this geometric solution \cite{Yokota,ChoMurakami}.

The numerical computations were carried out using higher-precision approximations of the geometric roots than those displayed below.

\subsection{The figure-eight knot}

We first consider the figure-eight knot $4_1$. Let $K$ be the knot $4_1$.
This example is included mainly as a preliminary consistency check. The complement of $4_1$ is amphicheiral, and its Chern--Simons invariant vanishes. Hence its complex volume is represented by the real number
$$
\mathbb V(K)
=
\operatorname{Vol}(S^3\setminus K)
=
2.0298832128\ldots .
$$

For the even-dimensional representations, the values of $\Delta_{K,\rho_{n}}(t)$ at $t=1$ are
$$
-2,\quad
4,\quad
-98,\quad
9604,\quad
-3225800,\quad
3902251024,\quad
-17358727595552,\quad
281300810154928144
$$
for $n=2,4,6,8,10,12,14,16$, respectively. 
For the odd-dimensional representations, $\Delta_{K,\rho_n}(t)$ has a factor $t-1$. After removing this factor, the values
$
\left.
\frac{\Delta_{K,\rho_n}(t)}{t-1}
\right|_{t=1}$
are
$$
3,\quad
-28,\quad
1440,\quad
-260736,\quad
165110400,\quad
-383313934080,\quad
3257341296168960,\quad
-100636318520821923840
$$
for $n=3,5,7,9,11,13,15,17$, respectively. 
Using these values, we obtain the following values of $Q_{K,n}$.

\newpage

\begin{table}[ht]%4_1の計算結果の表
\begin{center} 
\begin{tabular}{|c|ll||c|ll|} 
\hline
\multicolumn{1}{|c|}{$n$} & \multicolumn{2}{c||}{$Q_{K,n}$} &\multicolumn{1}{|c|}{$n$} & \multicolumn{2}{c|}{$Q_{K,n}$}\\ \hline\hline
$4$ & $3.93567...$ &  & $5$ & $2.68072...$ & \\ \hline
$6$ & $2.17759...$ &  & $7$ & $1.97712...$ & \\ \hline
$8$ & $1.93489...$ &  & $9$ & $1.96663...$ & \\ \hline
$10$ & $2.01277...$ & & $11$ & $2.04067...$ & \\ \hline
$12$ & $2.04544...$ &  & $13$ & $2.03822...$ & \\ \hline
$14$ & $2.03071...$ &  & $15$ & $2.02757...$ & \\ \hline
\end{tabular}
\caption{Numerical values of $Q_{K,n}$ for the knot $4_1$.} \label{value 4_1}
\end{center}
\end{table}

The values oscillate around
$$
\operatorname{Vol}(S^3\setminus K)
=
2.0298832128\ldots
$$
and appear to approach it. Moreover, the quantities occurring in this computation are real, so that no nontrivial imaginary part appears. This is consistent with
$$
2\pi^2\operatorname{CS}(S^3\setminus K)
\equiv0\pmod{\pi^2}.
$$

Thus the figure-eight knot gives a simple example consistent with our conjecture. However, it does not test the essential new feature of the complexified formula, namely the appearance of a nonzero Chern--Simons term. We therefore turn to the knot $5_2$.

\subsection{The hyperbolic knot with five crossings}

Let $K$ be the knot $5_2$. We explain the computation in detail in this subsection.

A Wirtinger presentation of the knot group is
$$
G(K)
=
\langle
a,b
\mid
b[a,b][a,b]=[a,b]a
\rangle,
$$
where
$$
[a,b]=aba^{-1}b^{-1},
$$
and the generators $a$ and $b$ correspond to the loops shown in Figure~\ref{5_2}.
\vspace{-5mm}
\begin{figure}[htbp]
\centering
\captionsetup{skip=-10pt}
\includegraphics[width=4cm]{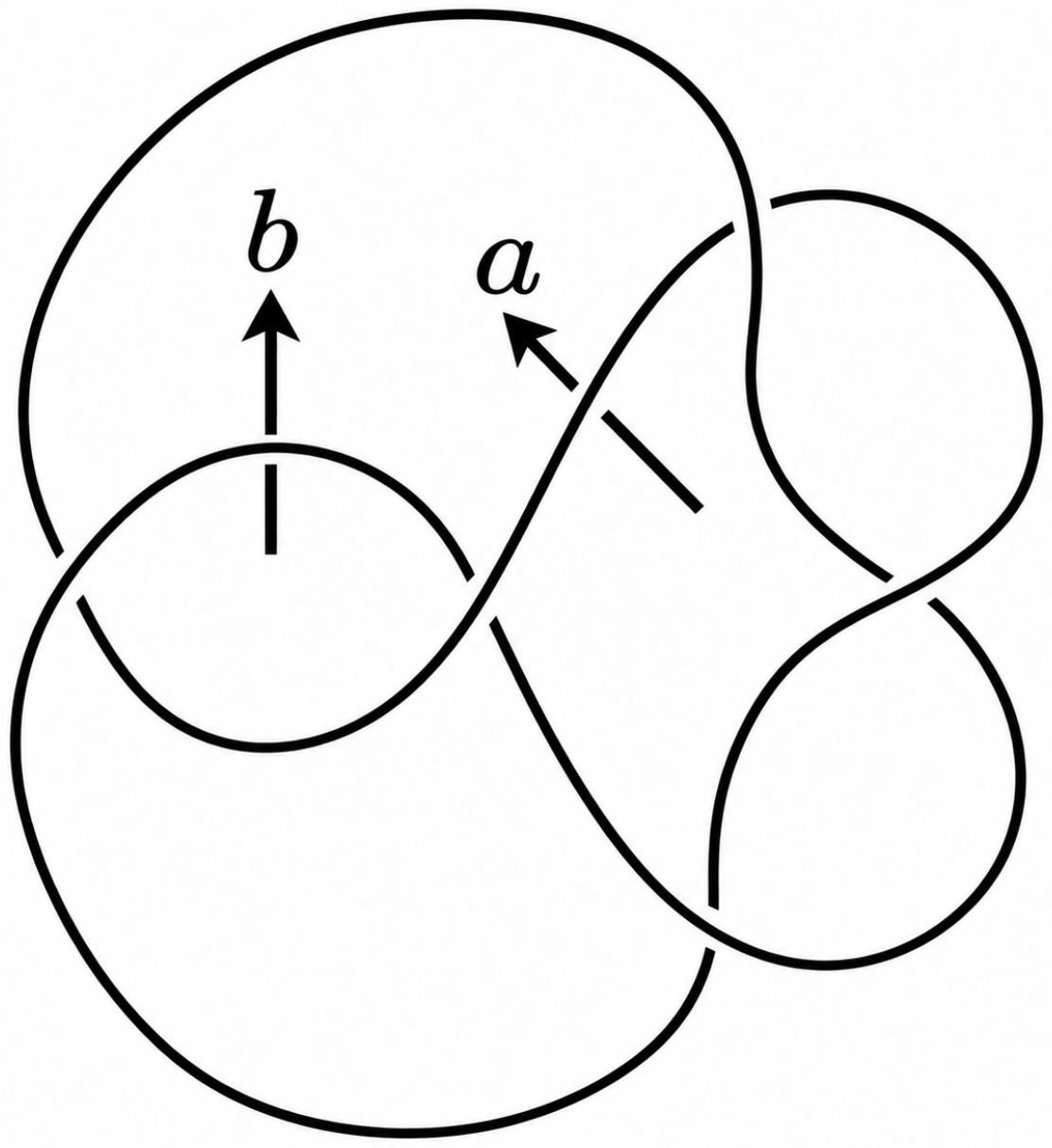}
\caption{The knot $5_2$.}
\label{5_2}
\end{figure}

We use the following lift of the holonomy representation:
$$
\rho:G(K)\longrightarrow SL(2,\bC),
$$

$$
\rho(a)
=
\begin{pmatrix}
1&0\\
1&1
\end{pmatrix},
\qquad
\rho(b)
=
\begin{pmatrix}
1&-1/x^2\\
0&1
\end{pmatrix},
$$
where $x$ is the root of
\[
x^3-x-1=0
\]
corresponding to the geometric representation, numerically given by
\[
x\approx -0.66235+0.56227i.
\]

Let
$$
\rho_n=\sigma_n\circ\rho.
$$
By Lemma~\ref{rho_n}, we have
$$
\rho_n(a)=A_n=[a_{ij}],
\qquad
\rho_n(b)=B_n=[b_{ij}],
$$
where
$$
a_{ij}
=
\begin{cases}
\displaystyle
\frac{(j-1)!}{(i-1)!(j-i)!}(-1)^{i+j},
& \text{if }i\leq j,\\[3mm]
0,
& \text{if }i>j,
\end{cases}
$$
and
$$
b_{ij}
=
\begin{cases}
0,
& \text{if }i<j,\\[2mm]
\displaystyle
\frac{(n-j)!}{(n-i)!(i-j)!}
x^{-2(i-j)},
& \text{if }i\geq j.
\end{cases}
$$
Put
$$
AB_n
:=
[A_n,B_n]
=
A_nB_nA_n^{-1}B_n^{-1}.
$$
The twisted Alexander polynomial associated with $\rho_n$ is written as
$$
\Delta_{K,\rho_n}(t)
=
\frac{N_n(t)}{D_n(t)},
$$
where
$$
N_n(t)
=
\det\bigl(
-t^{-1}B_n^{-1}(AB_n+E_n)+B_n^{-1}AB_nB_n+E_n+AB_n\\
-t(AB_n+E_n)AB_nB_n
\bigr)
$$
and
$$
D_n(t)
=
\det(tA_n-E_n).
$$
Since $A_n$ is upper triangular with all diagonal entries equal to $1$, we have
$$
D_n(t)=(t-1)^n.
$$

By the following, we can compute $Q_{K,n}$ directly from the numerators $N_n(t)$. 

\begin{prop}
There is an equality
$$
Q_{K,n}
=
\frac{\pi}{2}
\log
\left(
\frac{
N_{n-2}(1)N_{n+2}(1)
}{
N_n(1)^2
}
\right).
$$
\end{prop}
\begin{proof}
Recall that
$$
A_{K,2k}(t)
=
\frac{\Delta_{K,\rho_{2k}}(t)}
     {\Delta_{K,\rho_2}(t)}
$$
and
$$
A_{K,2k+1}(t)
=
\frac{\Delta_{K,\rho_{2k+1}}(t)}
     {\Delta_{K,\rho_3}(t)}.
$$
Because $n-2$, $n$, and $n+2$ have the same parity, the normalization factors cancel in
$$
\frac{
A_{K,n-2}(t)A_{K,n+2}(t)
}{
A_{K,n}(t)^2
}.
$$
The factors $(t-1)^n$ also cancel.
Consequently, we have 
$$
\frac{
A_{K,n-2}(t)A_{K,n+2}(t)
}{
A_{K,n}(t)^2
}
=
\frac{
N_{n-2}(t)N_{n+2}(t)
}{
N_n(t)^2
}.
$$
Substituting $t=1$  into this equation gives the desired equation.
\end{proof}

The numerical values obtained using Mathematica are shown in Table~\ref{value 5_2}. 

\begin{table}[H]%5_2の計算結果の表
\begin{center} 
\begin{tabular}{|c|ll||c|ll|} 
\hline
\multicolumn{1}{|c|}{$n$} & \multicolumn{2}{c||}{$Q_{K,n}$} &\multicolumn{1}{|c|}{$n$} & \multicolumn{2}{c|}{$Q_{K,n}$}\\ \hline\hline
$4$ & $3.90921...$ & + $(1.21314...) i$ & $5$ & $2.06464...$ & + $(2.28856...) i$\\ \hline
$6$ & $2.00009...$ & + $(3.60568...) i$ & $7$ & $3.12694...$ & + $(3.86417...) i$\\ \hline
$8$ & $3.52256...$ & + $(2.85486...) i$ & $9$ & $2.7451... $ & + $(2.46852...) i$\\ \hline
$10$ & $2.41642...$ & + $(3.03596...) i$ & $11$ & $2.79327...$ & + $(3.31223...) i$\\ \hline
$12$ & $3.03141...$ & + $(3.09112...) i$ & $13$ & $2.90802...$ & + $(2.88221...) i$\\ \hline
$14$ & $2.73082...$ & + $(2.94575...) i$ & $15$ & $2.75758...$ & + $(3.08777...) i$\\ \hline
$16$ & $2.86697... $ & + $(3.08451..) i$ & $17$ & $2.87796... $ & + $ (3.00243...) i$\\ \hline
$18$ & $2.81791... $ & + $ (2.98433...) i$ & $19$ & $2.79732... $ & + $ (3.02695...) i$\\ \hline
$20$ & $2.82645... $ & + $ (3.04734...) i$ & $21$ & $2.84518... $ & + $(3.02826...) i$\\ \hline
$22$ & $2.83331... $ & + $(3.01192...) i$ & $23$ & $2.81966... $ & + $(3.01876...) i$\\ \hline
$24$ & $2.82309... $ & + $ (3.02978..) i$& $25$ & $2.83171... $ & + $ (3.02857...) i$\\ \hline
$26$ & $2.83186...$ & + $ (3.02199...) i$ & $27$ & $2.82699... $ & + $(3.02108...) i$\\  \hline
$28$ & $2.82571... $ & + $ (3.02459...) i$ & & &\\ \hline
\end{tabular}
\caption{Numerical values of $Q_{K,n}$ for the knot $5_2$.} \label{value 5_2}
\end{center}
\end{table}

Using the potential-function computation described above, we obtain
\[
\operatorname{Vol}(S^3\setminus K)
=
2.82812\ldots
\]
and
\[
2\pi^2\operatorname{CS}(S^3\setminus K)
\equiv
3.02413\ldots
\pmod{\pi^2}.
\]
The numerical values in Table~\ref{value 5_2} therefore suggest that, for each parity,
$$
\lim_{n\to\infty}Q_{K,n}
=
\operatorname{Vol}(S^3\setminus K)
+
i\,2\pi^2\operatorname{CS}(S^3\setminus K)
\quad
\text{in }\bC/i\pi^2\bZ.
$$
To make the convergence more explicit, we also compute
$$
Q_{K,n}-\mathbb V(K).
$$
The resulting values are shown in Table~\ref{difference 5_2}.

\begin{table}[ht]
\begin{center}
\begin{tabular}{|c|ll||c|ll|}
\hline
\multicolumn{1}{|c|}{$n$}
& \multicolumn{2}{c||}{$Q_{K,n}-\mathbb V(K)$}
& \multicolumn{1}{|c|}{$n$}
& \multicolumn{2}{c|}{$Q_{K,n}-\mathbb V(K)$}
\\ \hline\hline

$4$  & $+1.08109...$  & $-(1.81099...)i$
& $5$  & $-0.76348...$ & $-(0.73557...)i$ \\ \hline

$6$  & $-0.82803...$ & $+(0.58155...)i$
& $7$  & $+0.29882...$  & $+(0.84004...)i$ \\ \hline

$8$  & $+0.69444...$  & $-(0.16927...)i$
& $9$  & $-0.08302...$ & $-(0.55561...)i$ \\ \hline

$10$ & $-0.41170...$ & $+(0.01183...)i$
& $11$ & $-0.03485...$ & $+(0.28810...)i$ \\ \hline

$12$ & $+0.20329...$  & $+(0.06699...)i$
& $13$ & $+0.07990...$  & $-(0.14192...)i$ \\ \hline

$14$ & $-0.09730...$ & $-(0.07838...)i$
& $15$ & $-0.07054...$ & $+(0.06364...)i$ \\ \hline

$16$ & $+0.03885...$  & $+(0.06038...)i$
& $17$ & $+0.04984...$  & $-(0.02170...)i$ \\ \hline

$18$ & $-0.01021...$ & $-(0.03980...)i$
& $19$ & $-0.03080...$ & $+(0.00282...)i$ \\ \hline

$20$ & $-0.00167...$ & $+(0.02321...)i$
& $21$ & $+0.01706...$  & $+(0.00413...)i$ \\ \hline

$22$ & $+0.00519...$  & $-(0.01221...)i$
& $23$ & $-0.00846...$ & $-(0.00537...)i$ \\ \hline

$24$ & $-0.00503...$ & $+(0.00565...)i$
& $25$ & $+0.00359...$  & $+(0.00444...)i$ \\ \hline

$26$ & $+0.00374...$  & $-(0.00214...)i$
& $27$ & $-0.00113...$ & $-(0.00305...)i$ \\ \hline

$28$ & $-0.00241...$ & $+(0.00046...)i$
& & & \\ \hline

\end{tabular}
\caption{Differences between $Q_{K,n}$ and the complex volume $\mathbb V(K)$ for the knot $5_2$.}
\label{difference 5_2}
\end{center}
\end{table}

In contrast to the figure-eight knot, the imaginary parts in Table~\ref{value 5_2} are nontrivial and appear to approach the Chern--Simons term. This is the first example in our computations in which the complex nature of the conjectured limit becomes visible.

\subsection{Hyperbolic knots with six crossings}

We next apply the same procedure to the three hyperbolic knots with six crossings, $6_1$, $6_2$, and $6_3$. Since the computational method was described in detail for $5_2$, we give only the data required for each knot and the resulting numerical values.

\subsubsection{The knot $6_1$}
\vspace{-10mm}
\begin{figure}[htbp]
\centering
\captionsetup{skip=-10pt}
\includegraphics[width=4.5cm]{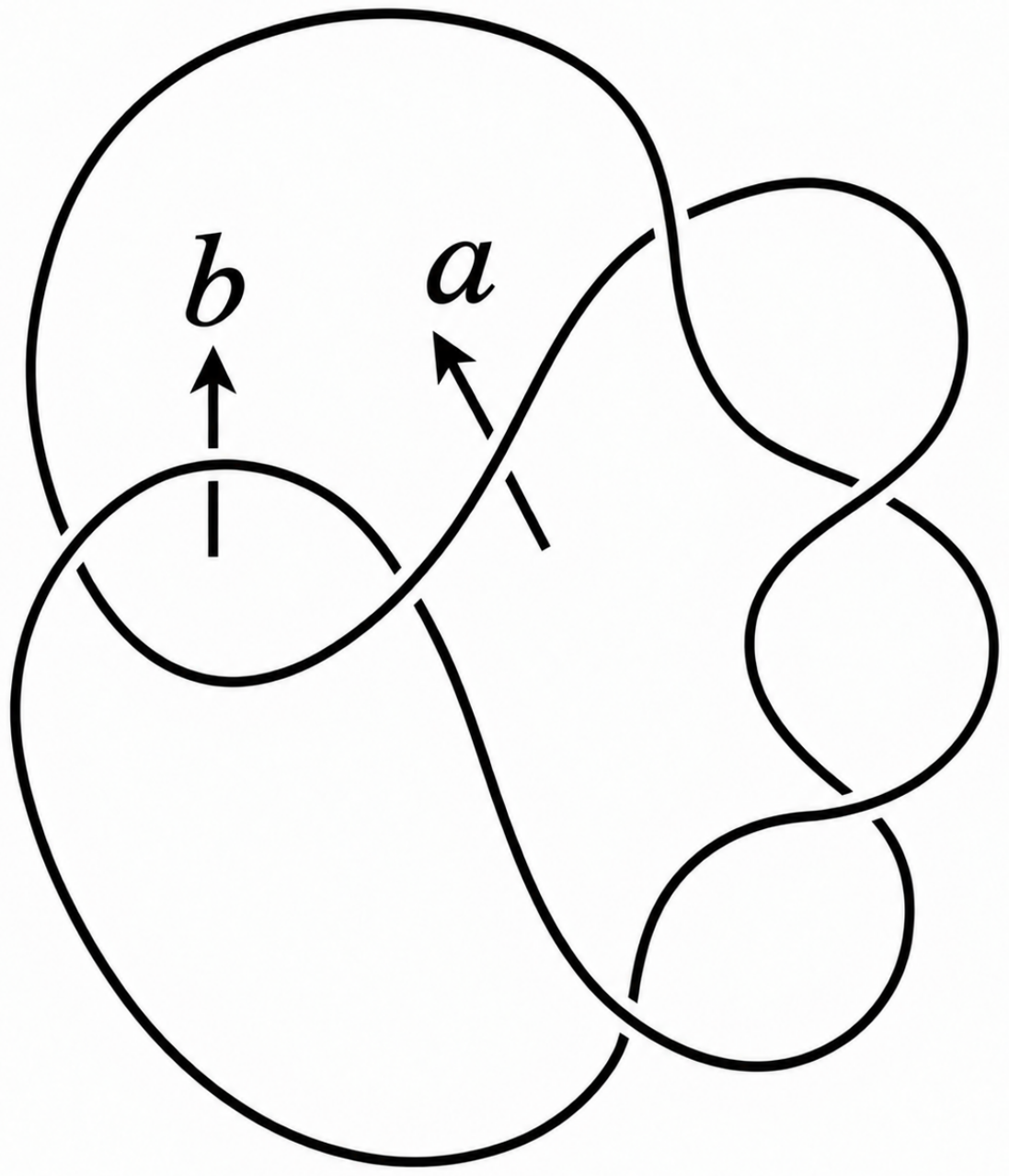}
\caption{The knot $6_1$.}
\label{6_1}
\end{figure}

Let $K$ be the knot $6_1$. A Wirtinger presentation is
$$
G(K)
=
\langle
a,b
\mid
[a,b][a,b]
=
b[a,b][a,b]a
\rangle.
$$
We use the lift of the holonomy representation
$$
\rho(a)
=
\begin{pmatrix}
1&0\\
1&1
\end{pmatrix},
\qquad
\rho(b)
=
\begin{pmatrix}
1&-\dfrac{x-1}{x}\\
0&1
\end{pmatrix},
$$
where $x$ denotes the root of
$$
2x^4-5x^3+6x^2-3x+1=0
$$
corresponding to the geometric representation,numerically given by
\[
x\approx 0.27872+0.48342i.
\]

Writing
$$
\rho_n(a)=A_n=[a_{ij}],
\qquad
\rho_n(b)=B_n=[b_{ij}],
$$
we have
$$
a_{ij}
=
\begin{cases}
\displaystyle
\frac{(j-1)!}{(i-1)!(j-i)!}(-1)^{i+j},
& \text{if }i\leq j,\\[3mm]
0,
& \text{if }i>j,
\end{cases}
$$
and
$$
b_{ij}
=
\begin{cases}
0,
& \text{if }i<j,\\[2mm]
\displaystyle
\frac{(n-j)!}{(n-i)!(i-j)!}
\left(\frac{x-1}{x}\right)^{i-j},
& \text{if }i\geq j.
\end{cases}
$$
Put
$$
AB_n=A_nB_nA_n^{-1}B_n^{-1}.
$$
Then
$$
\Delta_{K,\rho_n}(t)
=
\frac{N_n(t)}{(t-1)^n},
$$
where
$$
N_n(t)
=
\det
\left(
(E_n-tB_n)
(E_n+AB_n)
(E_n-tAB_nB_n)
-
tB_nAB_nAB_n
\right).
$$
Hence $Q_{K,n}$ is computed from
$$
Q_{K,n}
=
\frac{\pi}{2}
\log
\left(
\frac{
N_{n-2}(1)N_{n+2}(1)
}{
N_n(1)^2
}
\right).
$$
The numerical values of $Q_{K,n}$ are shown in Table~\ref{value 6_1}.

\begin{table}[ht]%6_1の計算結果の表
\begin{center} 
\begin{tabular}{|c|ll||c|ll|} 
\hline
\multicolumn{1}{|c|}{$n$} & \multicolumn{2}{c||}{$Q_{K,n}$} &\multicolumn{1}{|c|}{$n$} & \multicolumn{2}{c|}{$Q_{K,n}$}\\ \hline\hline
$4$ & $1.10035...$ & $+ (4.78702...) i$ &  $5$ & $1.84117...$ & $- (1.11143...) i $\\ \hline
$6$ & $5.56439...$ &  $- (1.99856...) i$ &   $7$ & $3.43741...$ & $- (4.81277...) i$\\ \hline
$8$ & $1.4012...$ & $- (3.29514...) i$ &  $9$ & $3.50659...$ & $- (1.63977...) i$\\ \hline
$10$ & $4.19526...$ & $- (3.53126...) i$ &  $11$ & $2.66498...$ & $- (3.95452...) i$\\ \hline
$12$ & $2.53582...$ & $- (2.44756...) i$ &  $13$ & $3.771...$ & $- (2.68371...) i$\\ \hline
$14$ & $3.40126...$ & $- (3.63418...) i$ &  $15$ & $2.64798...$ & $- (3.19445...) i$\\ \hline
$16$ & $3.16741...$ & $- (2.62951...) i$ & $17$ & $3.52798...$ & $- (3.15493...) i$\\ \hline
$18$ & $3.04638...$ & $ - (3.37271...) i$ & $19$ & $2.93679...$ & $ - (2.92699...) i$\\ \hline
$20$ & $3.32811...$ & $ - (2.9191...) i$ & $21$ & $3.27008...$ & $ - (3.23953...) i$\\ \hline
$22$ & $3.01150...$ & $ - (3.14172...) i$ &   $23$ & $3.13841...$ & $ - (2.9411...) i$\\ \hline
$24$ & $3.28178...$ & $ - (3.08068...) i$ &  $25$ & $3.14344... $ & $- (3.17666...) i$\\ \hline
$26$ & $3.08572... $ & $- (3.04515...) i$ & $27$ & $3.20519...$ & $ - (3.01979...) i$\\  \hline
$28$ & $3.20615...$ & $ - (3.12272...) i$ & & &\\ \hline
\end{tabular}
\caption{Numerical values of $Q_{K,n}$ for the knot $6_1$.} \label{value 6_1}
\end{center}
\end{table}

Since
$$
\operatorname{Vol}(S^3\setminus K)
=
3.16396\ldots
$$
and
$$
2\pi^2\operatorname{CS}(S^3\setminus K)
\equiv
-3.07886\ldots
\pmod{\pi^2},
$$
the numerical data are consistent with
$$
\lim_{n\to\infty}Q_{K,n}
=
\operatorname{Vol}(S^3\setminus K)
+
i\,2\pi^2\operatorname{CS}(S^3\setminus K)
\quad
\text{in }\bC/i\pi^2\bZ.
$$
See Table~\ref{difference 6_1} for the corresponding differences from the complex volume.

\begin{table}[ht]
\begin{center}
\begin{tabular}{|c|ll||c|ll|}
\hline
\multicolumn{1}{|c|}{$n$}
& \multicolumn{2}{c||}{$Q_{K,n}-\mathbb V(K)$}
& \multicolumn{1}{|c|}{$n$}
& \multicolumn{2}{c|}{$Q_{K,n}-\mathbb V(K)$}
\\ \hline\hline

$4$  & $-2.06361...$ & $+(7.86588...)i$
& $5$  & $-1.32279...$ & $+(1.96743...)i$ \\ \hline

$6$  & $+2.40043...$  & $+(1.08030...)i$
& $7$  & $+0.27345...$  & $-(1.73391...)i$ \\ \hline

$8$  & $-1.76276...$ & $-(0.21628...)i$
& $9$  & $+0.34263...$  & $+(1.43909...)i$ \\ \hline

$10$ & $+1.03130...$  & $-(0.45240...)i$
& $11$ & $-0.49898...$ & $-(0.87566...)i$ \\ \hline

$12$ & $-0.62814...$ & $+(0.63130...)i$
& $13$ & $+0.60704...$  & $+(0.39515...)i$ \\ \hline

$14$ & $+0.23730...$  & $-(0.55532...)i$
& $15$ & $-0.51598...$ & $-(0.11559...)i$ \\ \hline

$16$ & $+0.00345...$  & $+(0.44935...)i$
& $17$ & $+0.36402...$  & $-(0.07607...)i$ \\ \hline

$18$ & $-0.11758...$ & $-(0.29385...)i$
& $19$ & $-0.22717...$ & $+(0.15187...)i$ \\ \hline

$20$ & $+0.16415...$  & $+(0.15976...)i$
& $21$ & $+0.10612...$  & $-(0.16067...)i$ \\ \hline

$22$ & $-0.15246...$ & $-(0.06286...)i$
& $23$ & $-0.02555...$ & $+(0.13776...)i$ \\ \hline

$24$ & $+0.11782...$  & $-(0.00182...)i$
& $25$ & $-0.02052...$ & $-(0.09780...)i$ \\ \hline

$26$ & $-0.07824...$ & $+(0.03371...)i$
& $27$ & $+0.04123...$  & $+(0.05907...)i$ \\ \hline

$28$ & $+0.04219...$  & $-(0.04386...)i$
& & & \\ \hline

\end{tabular}
\caption{Differences between $Q_{K,n}$ and the complex volume
$\mathbb V(K)$ for the knot $6_1$.}
\label{difference 6_1}
\end{center}
\end{table}

\subsubsection{The knot $6_2$}
In this subsection, we use the diagram of \(6_2\) shown in Figure~\ref{6_2}, which is the mirror image of the diagram appearing in Rolfsen's knot table \cite{Rolfsen}.\vspace{-0mm}

\begin{figure}[htbp]
\centering
\captionsetup{skip=10pt}
\hspace{0mm}
\includegraphics[width=4.2cm]{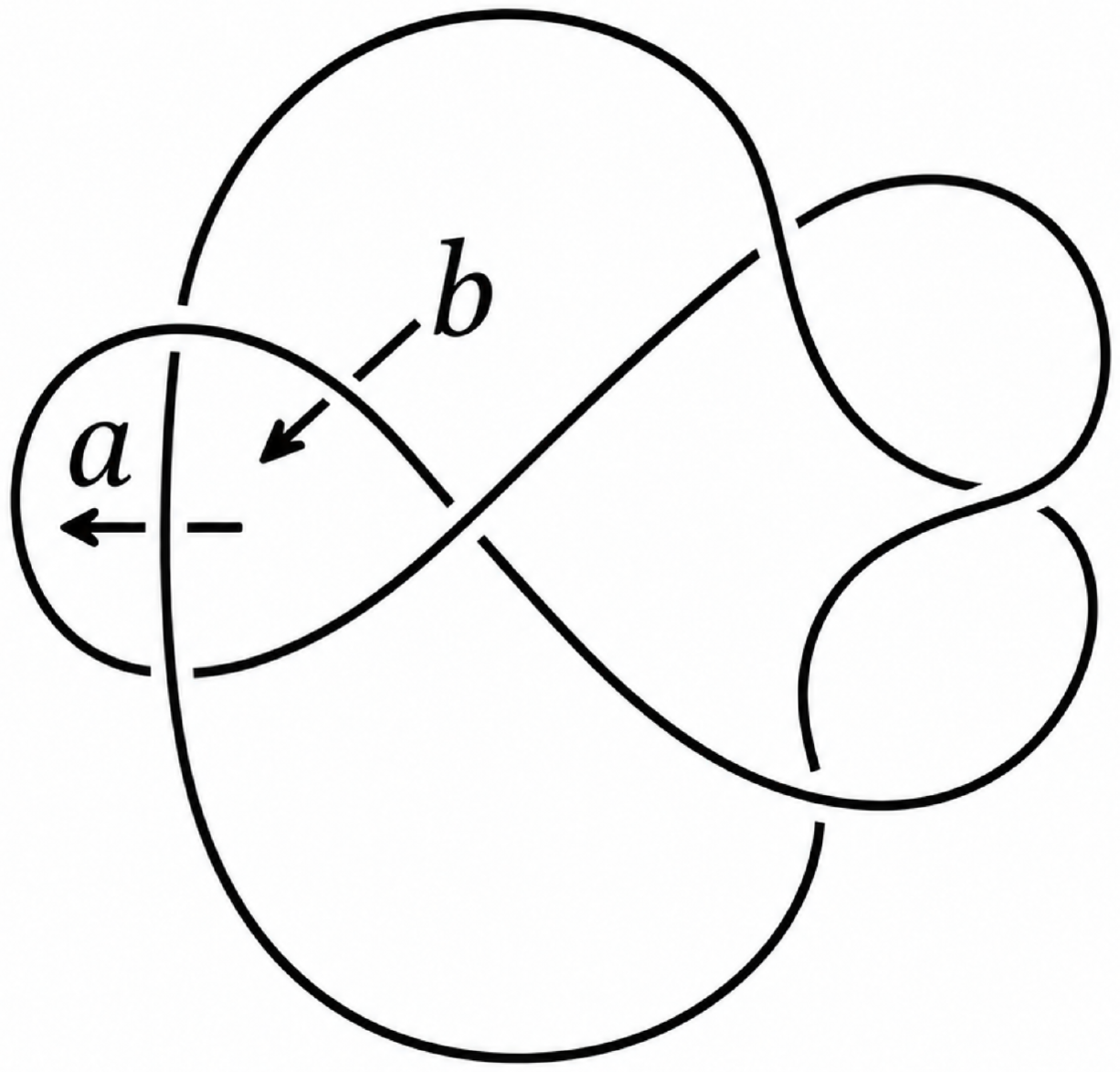}
\caption{The knot $6_2$.}
\label{6_2}
\end{figure}

Let $K$ be the knot $6_2$.
A Wirtinger presentation is
$$
G(K)
=
\left\langle
a,b
\ \middle|\
(aba^{-1})(bab^{-1})(aba^{-1})b
=
(bab^{-1})(aba^{-1})(bab^{-1})(aba^{-1})
\right\rangle.
$$
We use
$$
\rho(a)
=
\begin{pmatrix}
1&1\\
0&1
\end{pmatrix},
\qquad
\rho(b)
=
\begin{pmatrix}
1&0\\[1mm]
\dfrac{x^2}{(x-1)(1+x)^2}&1
\end{pmatrix},
$$
where $x$ denotes the root of
$$
1+2x-x^2-2x^3+x^5=0
$$
corresponding to the geometric representation, numerically given by
\[
x\approx -0.96491+0.62189i.
\]

Writing
$$
\rho_n(a)=A_n=[a_{ij}],
\qquad
\rho_n(b)=B_n=[b_{ij}],
$$
we have
$$
a_{ij}
=
\begin{cases}
0,
& \text{if }i<j,\\[2mm]
\displaystyle
\frac{(n-j)!}{(n-i)!(i-j)!}(-1)^{i+j},
& \text{if }i\geq j,
\end{cases}
$$
and
$$
b_{ij}
=
\begin{cases}
\displaystyle
\frac{(j-1)!}{(i-1)!(j-i)!}
\left(
\frac{x^2}{(1-x)(1+x)^2}
\right)^{j-i},
& \text{if }i\leq j,\\[3mm]
0,
& \text{if }i>j.
\end{cases}
$$

Put
$$
C_n=A_nB_nA_n^{-1},
\qquad
D_n=B_nA_nB_n^{-1}.
$$
Then
$$
\Delta_{K,\rho_n}(t)
=
\frac{N_n(t)}{(t-1)^n},
$$
where
$$
N_n(t)
=
\det
\left(
(E_n-tD_n)
(E_n+t^2C_nD_n)
(E_n+B_nA_n^{-1}-tC_n)
-
B_nA_n^{-1}
\right).
$$
The numerical values of $Q_{K,n}$ are shown in Table~\ref{value 6_2}.

\begin{table}[ht]%6_2の計算結果の表
\begin{center} 
\begin{tabular}{|c|ll||c|ll|} 
\hline
\multicolumn{1}{|c|}{$n$} & \multicolumn{2}{c||}{$Q_{K,n}$ } &\multicolumn{1}{|c|}{$n$} & \multicolumn{2}{c|}{$Q_{K,n}$ }\\ \hline\hline
$4$ & $4.68241...$ & $ + (3.26794...) i$ &  $5$ & $3.07457...$ & $ - (4.44537...) i$\\ \hline
$6$ & $4.04742...$ & $ - (3.09489...) i$ &   $7$ & $5.0858...$ & $ - ( 3.76851...) i$ \\ \hline
$8$ & $4.5338...$ & $ - (4.5069...) i$ &  $9$ & $4.01741...$ & $ - (4.06763...) i$\\ \hline
$10$ & $4.36261...$ & $ - (3.71779...) i$ &  $11$ & $4.60593...$ & $ - (3.96618...) i$\\ \hline
$12$ & $4.42781...$ & $ - (4.15257...) i$ &  $13$ & $4.28296...$ & $ - (4.01862...) i$\\ \hline
$14$ & $4.38351... $ & $- (3.90932...) i$ & $15$ & $4.4656... $ & $- (3.98257...) i$\\ \hline
$16$ & $4.41281... $ & $- (4.04489...) i$ & $17$ & $4.36561... $ & $- (4.00665...) i$\\ \hline
$18$ & $4.39322... $ & $- (3.97128...) i$ & & &\\ \hline
\end{tabular}
\caption{Numerical values of $Q_{K,n}$ for the knot $6_2$.} \label{value 6_2}
\end{center}
\end{table}

Since
$$
\operatorname{Vol}(S^3\setminus K)
=
4.40083\ldots
$$
and
$$
2\pi^2\operatorname{CS}(S^3\setminus K)
\equiv
-3.99704\ldots
\pmod{\pi^2},
$$
the numerical values again agree with the expected complex volume and are consistent with
$$
\lim_{n\to\infty}Q_{K,n}
=
\operatorname{Vol}(S^3\setminus K)
+
i\,2\pi^2\operatorname{CS}(S^3\setminus K)
\quad
\text{in }\bC/i\pi^2\bZ.
$$
See Table~\ref{difference 6_2} for the corresponding differences from the complex volume.
Note that our diagram is the mirror image of the representative in Rolfsen's table. This convention is relevant to the sign of the imaginary part of the complex volume.

\begin{table}[ht]
\begin{center}
\begin{tabular}{|c|ll||c|ll|}
\hline
\multicolumn{1}{|c|}{$n$}
& \multicolumn{2}{c||}{$Q_{K,n}-\mathbb V(K)$}
& \multicolumn{1}{|c|}{$n$}
& \multicolumn{2}{c|}{$Q_{K,n}-\mathbb V(K)$}
\\ \hline\hline

$4$  & $+0.28158...$ & $+(7.26498...)i$
& $5$  & $-1.32626...$ & $-(0.44833...)i$ \\ \hline

$6$  & $-0.35341...$ & $+(0.90215...)i$
& $7$  & $+0.68497...$ & $+(0.22853...)i$ \\ \hline

$8$  & $+0.13297...$ & $-(0.50986...)i$
& $9$  & $-0.38342...$ & $-(0.07059...)i$ \\ \hline

$10$ & $-0.03822...$ & $+(0.27925...)i$
& $11$ & $+0.20510...$ & $+(0.03086...)i$ \\ \hline

$12$ & $+0.02698...$ & $-(0.15553...)i$
& $13$ & $-0.11787...$ & $-(0.02158...)i$ \\ \hline

$14$ & $-0.01732...$ & $+(0.08772...)i$
& $15$ & $+0.06477...$ & $+(0.01447...)i$ \\ \hline

$16$ & $+0.01198...$ & $-(0.04785...)i$
& $17$ & $-0.03522...$ & $-(0.00961...)i$ \\ \hline

$18$ & $-0.00761...$ & $+(0.02576...)i$
& & & \\ \hline

\end{tabular}
\caption{Differences between $Q_{K,n}$ and the complex volume
$\mathbb V(K)$ for the knot $6_2$.}
\label{difference 6_2}
\end{center}
\end{table}

\subsubsection{The knot $6_3$}
We finally consider the knot $6_3$. 
In this subsection, we use the diagram of \(6_3\) shown in Figure~\ref{6_3}, which is the mirror image of the diagram appearing in Rolfsen's knot table\cite{Rolfsen}.

\begin{figure}[htbp]
\centering
\captionsetup{skip=-10pt}
\hspace{-5mm}
\includegraphics[width=5cm]{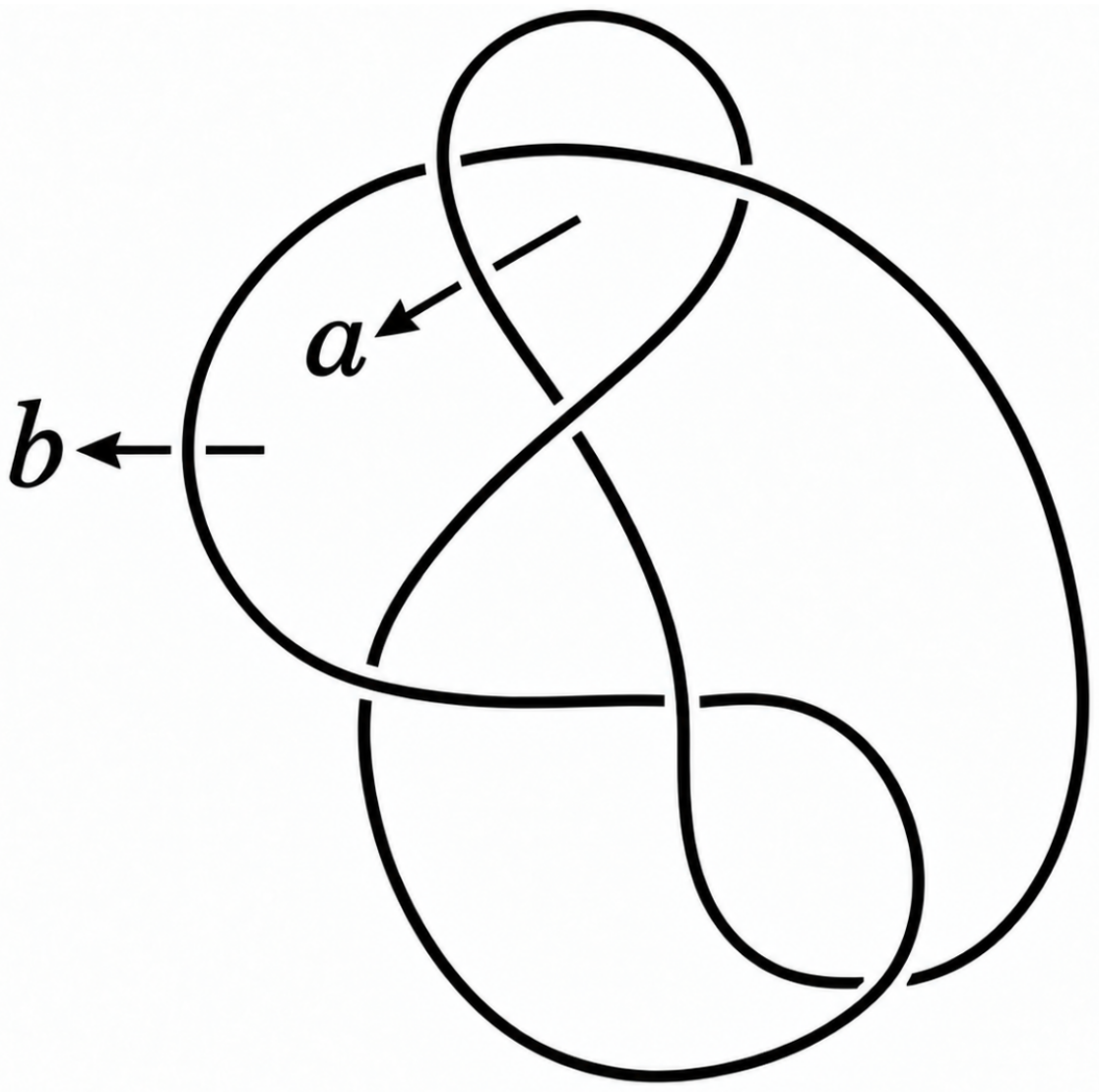}
\caption{The knot $6_3$.}
\label{6_3}
\end{figure}

Let $K$ be the knot $6_3$.
A Wirtinger presentation is
$$
G(K)
=
\left\langle
a,b
\ \middle|\
[b^{-1},a^{-1}]b[a^{-1},b^{-1}]a b^{-1}[a^{-1},b^{-1}]
=a^{-1}[b^{-1},a^{-1}]b[a^{-1},b^{-1}]a
\right\rangle.
$$
We use
$$
\rho(a)
=
\begin{pmatrix}
1&0\\
x(1-x)&1
\end{pmatrix},
\qquad
\rho(b)
=
\begin{pmatrix}
1&-\dfrac{1}{x}\\[1mm]
0&1
\end{pmatrix},
$$
where $x$ denotes the root of
$$
x^6-3 x^5+5 x^4-6 x^3+5 x^2-2 x+1=0
$$
corresponding to the geometric representation, numerically given by
\[
x\approx 0.15883+1.20014 i.
\]
Writing
$$
\rho_n(a)=A_n=[a_{ij}],
\qquad
\rho_n(b)=B_n=[b_{ij}],
$$
we have
$$
a_{ij}
=
\begin{cases}
\displaystyle
\frac{(j-1)!}{(j-i)!(i-1)!}(x^2-x)^{j-i},
& \text{if }i\leq j,\\[2mm]
0,
& \text{if }i> j,
\end{cases}
$$
and
$$
b_{ij}
=
\begin{cases}
0,
& \text{if }i< j,\\[3mm]
\displaystyle
\frac{(n-j)!}{(i-j)!(n-i)!}
x^{j-i},
& \text{if }i\geq j.
\end{cases}
$$
Put
$$
C_n=A_n^{-1}B_n^{-1}A_nB_n,
\qquad
D_n=B_n^{-1}A_n^{-1}B_nA_n.
$$
Then
$$
\Delta_{K,\rho_n}(t)
=
\frac{N_n(t)}{(t-1)^n},
$$
where
\[
N_n(t)
=
\det\left(
\begin{aligned}
 (E_n - t^{-1} A_n^{-1})  
  D_n  (t^{-2} (t B_n - E_n) A_n^{-1} B_n^{-1} (t A_n& - E_n) - 
    t B_n C_n A_n B_n^{-1} + E_n) \\
&   - t^{-1} D_n B_n C_n A_n B_n^{-1}  A_n^{-1} B_n^{-1}
\end{aligned}
\right)
\]
The numerical values of $Q_{K,n}$ are shown in Table~\ref{value 6_3}.

\begin{table}[ht]%6_3の計算結果の表
\begin{center} 
\begin{tabular}{|c|ll||c|ll|} 
\hline
\multicolumn{1}{|c|}{$n$} & \multicolumn{2}{c||}{$Q_{K,n}$} &\multicolumn{1}{|c|}{$n$} & \multicolumn{2}{c|}{$Q_{K,n}$}\\ \hline\hline
$4$ & $7.34561... $ & $+(1.2752...)\times10^{-16} i$ & $5$ & $5.91042... $ & $+(3.26134...)\times10^{-16} i$ \\ \hline
$6$ & $5.54295...$ & $+(3.45705...)\times10^{-16} i$ & $7$ & $5.55848... $ & $+(2.30278...)\times10^{-16} i$\\ \hline
$8$ & $5.6589...$ & $+(2.36723...)\times10^{-16} i$ & $9$ & $5.71527... $ & $+(3.03229...)\times10^{-16} i$\\ \hline
$10$ & $5.71617... $ & $+(2.3070...)\times10^{-16} i$& $11$ & $5.6988...$ & $+(4.15916...)\times10^{-16} i$\\ \hline
$12$ & $5.68932... $ & $+(3.18744...)\times10^{-16} i$ & $13$ & $5.68951...$ & $+(2.55525...)\times10^{-16} i$\\ \hline
$14$ & $5.69229... $ &  $+(4.28996...)\times10^{-16} i$& $15$ & $5.69363...$ & $+(4.89375...)\times10^{-16} i$\\ \hline
\end{tabular}
\caption{Numerical values of $Q_{K,n}$ for the knot $6_3$.} \label{value 6_3}
\end{center}
\end{table}

For the knot $6_3$,
$$
\operatorname{Vol}(S^3\setminus K)
=
5.69302\ldots
$$
and
$$
2\pi^2\operatorname{CS}(S^3\setminus K)
\equiv
0
\pmod{\pi^2}.
$$
Hence its complex volume is represented by
$$
\mathbb V(K)
=
5.69302\ldots .
$$
The values in Table~\ref{value 6_3} appear to approach this number.

\begin{table}[ht]%6_3の差の計算結果の表
\begin{center} 
\begin{tabular}{|c|ll||c|ll|} 
\hline
\multicolumn{1}{|c|}{$n$} & \multicolumn{2}{c||}{$Q_{K,n}-\mathbb V(K)$} &\multicolumn{1}{|c|}{$n$} & \multicolumn{2}{c|}{$Q_{K,n}-\mathbb V(K)$}\\ \hline\hline
$4$ & $+1.65259...$ & $+(1.2752...)\times10^{-16} i$ & $5$ & $+0.2174...$ & $+(3.26134...)\times10^{-16} i$  \\ \hline
$6$ & $-0.150067...$ & $+(3.45705...)\times10^{-16} i$ & $7$ & $-0.134542...$ & $+(2.30278...)\times10^{-16} i$ \\ \hline
$8$ & $-0.0341174...$ & $+(2.36723...)\times10^{-16} i$ & $9$ & $+0.0222527...$ & $+(3.03229...)\times10^{-16} i$\\ \hline
$10$ & $+0.0231439...$ &$+(2.3070...)\times10^{-16} i$ & $11$ & $+0.00577945...$ & $+(4.15916...)\times10^{-16} i$ \\ \hline
$12$ & $-0.00370422...$ & $+(3.18744...)\times10^{-16} i$ & $13$ & $-0.00351105...$ & $+(2.55525...)\times10^{-16} i$\\ \hline
$14$ & $-0.000729367...$ & $+(4.28996...)\times10^{-16} i$ & $15$ & $+0.000607015...$ & $+(4.89375...)\times10^{-16} i$\\ \hline
\end{tabular}
\caption{Differences between $Q_{K,n}$ and the complex volume $\mathbb V(K)$ for the knot $6_3$.} \label{difference 6_3}
\end{center}
\end{table}

Table~\ref{difference 6_3} makes the convergence particularly clear. The absolute value of the real part of the difference decreases to less than $10^{-3}$ for $n=14$ and $15$. The imaginary parts are of order $10^{-16}$ throughout the computation and may be regarded as numerical errors.
Thus the numerical data for $6_3$ are consistent with
$$
\lim_{n\to\infty}Q_{K,n}
=
\mathbb V(K)
\quad
\text{in }\bC/i\pi^2\bZ.
$$

Since the Chern--Simons invariant vanishes in this case, this provides another example, complementary to $5_2$, $6_1$, and $6_2$, in which the imaginary part predicted by the complex volume is zero.

Taken together, the computations for $4_1$, $5_2$, $6_1$, $6_2$, and $6_3$ provide numerical evidence for the conjecture stated in the Introduction. In the examples with nonzero Chern--Simons invariant, both the real and imaginary parts of $Q_{K,n}$ appear to approach the corresponding components of the complex volume. In the amphicheiral examples considered here, the Chern--Simons term vanishes and the computed imaginary parts are consistent with zero.

\end{document}